\documentclass[12pt,a4paper,twoside]{amsart}
\usepackage[top=1in, bottom=1in, marginparwidth=0.8in, inner=1in, outer=1in]{geometry}

\usepackage{preamble}

\newcommand{\Mor}{\mathrm{Mor}}

\DeclarePairedDelimiter{\ceil}{\lceil}{\rceil}
\newcommand{\gr}{\mathrm{gr}}

\usepackage{scalerel}
\renewcommand{\odot}{{\scalerel*{\includegraphics{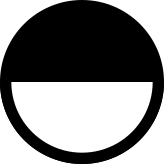}}{\bullet}}}

\newcommand{\iRes}{\scalerel*{\includegraphics[scale=0.25]{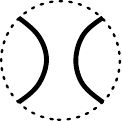}}{\bigoplus}}
\newcommand{\oRes}{\scalerel*{\rotatebox[origin=c]{-90}{\includegraphics[scale=0.25]{infTangle}}}{\bigoplus}}

\usepackage{tikz}
\usetikzlibrary{calc}

\author{Mihai Marian}
\address{Department of Mathematics \\ University of British Columbia}
\email{mihmar@math.ubc.ca}
\title{A remark on the Lewark--Zibrowius invariant}

\begin{document}
\maketitle

\begin{abstract} We prove a conjecture about the concordance invariant $\vartheta$, defined in a recent paper by Lewark and Zibrowius. This result simplifies the relation between $\vartheta$ and Rasmussen's $s$-invariant. The proof relies on Bar-Natan's tangle version of Khovanov homology or, more precisely, on its distillation in the case of 4-ended tangles into the immersed curve theory of Kotelskiy--Watson--Zibrowius.
\end{abstract}

\section{Introduction}

Lewark and Zibrowius define two new families of smooth concordance invariants,
\[\{\vartheta_c \co \mc{C}_\text{sm} \ra \Z\} \quad \text{and} \quad \{\vartheta'_c \co \mc{C}_\text{sm} \ra \Z \cup \{\infty\} \},\] parametrized by a prime $c$ in \cite{LZ24}. These invariants exploit the following linearity property of Rasmussen's invariant in characteristic $c$. Given a knot $K \subset S^3$ and a pattern $P \subset D^2 \times S^1$ of wrapping number 2, the function
\[t \mapsto s_c(P_t(K))\] 
is the restriction to $\Z$ of a piecewise affine function $\R \ra \R$ of slope 1 or 0 that has at most one jump discontinuity. If the winding number of $P$ is $\pm2$ then the function has slope 1, otherwise the winding number and slope are 0 and, in this latter case, the function does have a jump discontinuity. In the case of winding number $\pm2$, the invariant $\vartheta'_c(K)$ is defined to be the value of $t$ for which 
\[s_c(P_{\vartheta_c'(K)}(K)) = s_c(P_{\vartheta_c'(K)-1}(K)),\]
if it exists. If no such value exists because the piecewise affine function is affine, then $\vartheta'_c(K) := \infty$. Not only do Lewark--Zibrowius prove that $\vartheta_c$ and $\vartheta'_c$ are concordance invariants and that $\vartheta_c$ is a homomorphism $\mc{C}_\text{sm} \ra \Z$, but they also show that $\vartheta_c$ is a genuinely new invariant, in that it is not simply a multiple of $s_c$, in contrast to the $\tau$-invariant \cite[\S2.2]{LZ24}.

The knots $K$ with $\vartheta'_c(K) \neq \infty$ are of particular interest, and they are called $\vartheta_c$-rational. We establish here a conjecture on the expected simplicity of $\vartheta'_c$:

\begin{thm}[\!\!\protect{\cite[Conjecture 2.24]{LZ24}}] \label{conj} If $K$ is a $\vartheta_c$-rational knot, then $\vartheta_c(K) = 0$.
\end{thm}

Since $\vartheta_c$ agrees with $\vartheta_c'$ on the class of $\vartheta_c$-rational knots \cite[Theorem 2.23]{LZ24}, it follows that the second family of invariants $\{\vartheta'_c\}$ contains no more information than a single $\Z/2\Z$-valued invariant. A consequence noted by Lewark--Zibrowius in \cite[p. 250]{LZ24} is the following simplification of their Theorem 2.23:

\begin{cor} Let $K \subset S^3$ be a $\vartheta_c$-rational knot and let $P$ be a pattern with wrapping number 2 and winding number $\pm 2$. Then
	\[s_c(P(K)) = s_c(P_{-\vartheta_c(K)}(U)) = s_c(P_0(U)).\] \qed
\end{cor}

Our argument uses the immersed curve theory of 4-ended tangles, constructed in \cite{KWZ19} as a specialization of the theory developed in \cite{BN05}, and a property of Lee's homology \cite{Lee05}.

\textbf{Acknowledgment.} I extend my gratitude to Claudius Zibrowius and Liam Watson for generously sharing their feedback and suggestions.

\section{Background}\label{sec:background}

Tangles are considered modulo isotopy fixing the endpoints. Let $(K, *)$ be a pointed oriented knot and let $T_K$ be the 4-ended tangle obtained by taking a copy of the long knot $K \setminus \{*\}$ together with its Seifert push-off, as in \cref{fig:figEightDouble}. We generally also orient our tangles and mark an endpoint, as required for the theory in \cite{KWZ19}.

\begin{figure}[h]
	\centering
	\includegraphics[height=0.15\textheight]{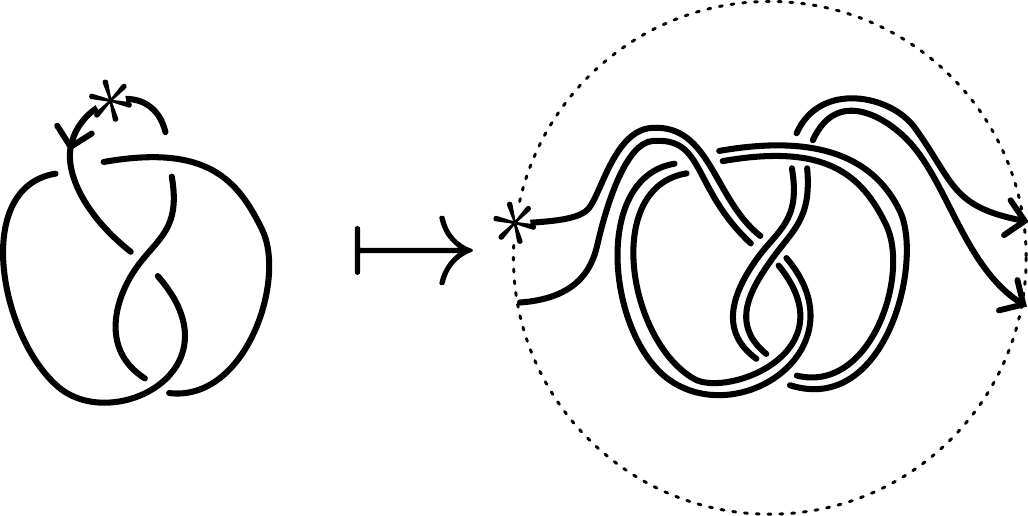}
	\caption{A pointed oriented knot $(K, \ast)$ and its associated double $T_K$.}
	\label{fig:figEightDouble}
\end{figure}

To specify notation for the cut-and-paste procedures used, let $n \in \Z \cup \{\infty\}$. First, the rational $n$-tangle $Q_n$ is the one in \cref{fig:nTangle} for $n>0$. If $n<0$, then $Q_n = mQ_{-n}$, where $m$ denotes the mirror. And if $n = 0, \infty$, we set $Q_0 = \oRes$ and $Q_\infty = \iRes$. Second, given two 4-ended tangles $T_1$ and $T_2$, the link $\mc{L}(T_1, T_2)$ is obtained by identifying endpoints as in \cref{fig:glue} below. Finally, let the $n$-closure $T(n)$ of a 4-ended tangle $T$ be $\mc{L}(T, Q_{-n})$.
By convention, diagrams for the tangle $T_K$ are chosen so that their $\infty$-closure is the unknot, and the tangle is oriented compatibly with the 0-closure, as in \cref{fig:figEightDouble}.

\begin{figure}[ht]
	\centering
	\begin{minipage}{.5\textwidth}
		\centering
		\includegraphics[height = 0.1\textheight]{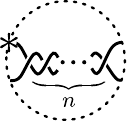}
		\caption{The tangle $Q_n$} \label{fig:nTangle}
	\end{minipage}%
	\begin{minipage}{.5\textwidth}
		\centering
		\includegraphics[height=0.1\textheight]{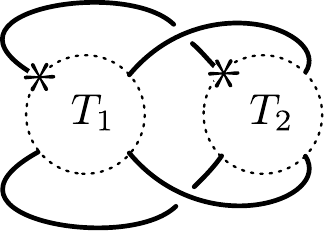}
		\caption{The link $\mc{L}(T_1, T_2)$.}	
		\label{fig:glue}
	\end{minipage}
\end{figure}

\subsection{Bar-Natan homology}

The Bar-Natan homology of a link is a version of Khovanov homology \cite{Kho00} defined in \cite{BN05} with coefficients in the field with two elements $\F_2$, and later extended as a theory with coefficients in any prime field in \cite{MTV07}. It has been observed that varying the field characteristic results in interesting differences \cite{LZ21}, so let $\F_c$ be the prime field of characteristic $c$ (in particular, $\F_0 = \Q$). We use the set-up in \cite[\S3]{KWZ19}.

Given a link $L$, its Bar-Natan homology is a bigraded $\F_c[H]$-module $\BN(L; \F_c)$, where $H$ is a formal variable that lowers the secondary (quantum) grading by $2$. The shift operators for the homological and quantum gradings are denoted using square and curly brackets, respectively. For example,
\[\BN(L; \F_c) \{-1\}\]
is the Bar-Natan homology of $L$ with coefficients in $\F_c$, but with quantum gradings formally reduced by 1.

If the link $L$ is pointed, then there is a reduced theory $\wt{\BN}(L; \F_c)$, which is related to unreduced Bar-Natan homology by a short exact sequence of bigraded $\F_c[H]$-complexes:
\begin{equation} \label{eq:redBNLES}\begin{tikzcd}
	0 \rar		&\wt{\CBN}(D; \F_c)\{-1\} \rar		&\CBN(D; \F_c) \rar	& \wt{\CBN}(D;\F_c)\{1\} \rar		& 0,
\end{tikzcd}\end{equation}
where $D$ is a choice of diagram for $L$.

\begin{notn} Free summands of the bigraded $\F_c[H]$-module $\wt{\BN}(L; \F_c)$ are called towers. The grading of a tower refers to the grading of a corresponding free generator.
\end{notn}

\subsection{Lee's deformation}\label{sec:LeeBN}

In \cite{Ras10}, Rasmussen uses the work in \cite{Lee05} to define the $s$-invariant of a knot. While the $s$-invariant can also be defined for links, as in \cite{BW08} and \cite{Par12}, this construction is not used as much, and Lewark--Zibrowius arrange so that their work only deals with $s$-invariants of knots. This subsection recalls an aspect of the definition of the $s$-invariant for links in \Cref{lem:towHomGrad} below. This result is known to the experts and is the main observation needed to prove \cref{conj}. See also \cite[Proposition 4.3]{Lee05}.

\begin{lem}\label{lem:towHomGrad} Let $L$ be an oriented 2-component pointed link.  If $\lk(L) \neq 0$, then there is a unique tower $\F_c[H] \hookrightarrow \wt{\BN}(L;\F_c)$ in homological grading 0. Otherwise, if $\lk(L) = 0$, then both towers have homological grading 0. 
\end{lem}

\begin{proof} The idea is that, by setting $H=1$ in the chain complex $\CBN(L; \F_c)$, we obtain a chain complex $\mathit{fCBN}(L; \F_c)$ that is no longer bigraded, but rather homologically graded and quantum filtered. Courtesy of the filtration, there is an induced spectral sequence
\[\mathit{fCBN}(L; \F_c) \rightrightarrows H_*(\mathit{fCBN}(L; \F_c)).\]
Theorem 2.2 of \cite{LS14} shows that the vector space $H_*(\mathit{fCBN}(L; \F_c))$ is $4$-dimensional, and there is a canonical identification between the set of orientations on $L$ and a set of generators of $H_*(\mathit{fCBN}(L; \F_c))$. To understand this identification, note that each orientation on $L$ determines an oriented resolution of a diagram for $L$. Lee's argument applies in this context to show that each generator of $H_*(\mathit{fCBN}(L; \F_c))$ is the homology class of an algebra element assigned to an oriented resolution of $L$ by the TQFT defining $\mathit{fCBN}$; see \cite[Theorem 4.2]{Lee05} or \cite[\S2.4]{Ras10} for the construction and \cite[Theorem 2.2]{LS14} for the applicability of Lee's work in this slightly different context.

Now, as explained in \cite[Proposition 3.8]{KWZ19}, the components of the differential $\del_{\CBN(L)}$ that are given by $1 \mapsto H^l$ induce differentials on the $l^{\text{th}}$ page of the spectral sequence above, and this implies that
\[\BN(L; \F_c) \iso (\F_c[H])^{\oplus 4} \oplus \mathrm{Tors},\]
where the towers in ${\BN}(L; \F_c)$ correspond to the generators of $H_*(\mathit{fCBN}(L; \F_c)$. Moreover it follows from the short exact sequence \cref{eq:redBNLES} that there is a 2-to-1 correspondence that preserves homological grading between the towers of $\BN(L)$ and the towers of $\wt{\BN}(L)$.

Finally, fix an oriented diagram $(D, \fk{o}_0)$ for $L$, where $\fk{o}_0$ is the orientation on $D$ induced from $L$. Let $n_+(\fk{o}_0)$ and $n_-(\fk{o}_0)$ be the number of positive and negative crossings in $(D, \fk{o}_0)$. Pick a component $K$ of $L$ and let $\fk{o}_1$ be the orientation on $D$ which is obtained by reversing the orientation on $K$. Then the number of negative crossings in $(D, \fk{o}_1)$ is
	\[n_-(\fk{o}_1) = n_-(\fk{o}_0) + 2 \lk(L) .\]
It follows that, while the oriented resolution of $(D, \fk{o}_0)$ lies in homological grading 0, the $\fk{o}_1$-oriented resolution $D^{\fk{o}_1}$ lies in homological grading $2\lk(L)$.
\end{proof}

\subsection{The immersed curve theory}

In \cite{KWZ19}, two equivalent invariants of pointed 4-ended oriented tangles are defined:
\[\begin{split}
	T	& \mapsto \Rd(T; \F_c) \in \cat{Mod}^\mc{B}\\
	T	& \mapsto \wt{\BN}(T; \F_c) \in \cat{Fuk}(S^2_{4, *}).
\end{split}\]
The first produces type D structures over the Bar-Natan algebra $\mc{B}$, which we will describe in \cref{sec:epilogue}. The second lands in the (partially wrapped) Fukaya category of $S^2$, punctured at four points, one of which is marked $\ast$. In other words, $\wt{\BN}(T; \F_c)$ is an immersed curve in $S^2_{4,\ast}$, possibly carrying a non-trivial local system. This possibility does not occur for non-compact curves, which are the only curves of interest in what follows. Moreover, the invariants are bigraded in an appropriate sense. Our main tool is the following pairing theorem.

\begin{thm}[\!\! {\cite[Theorem 7.2]{KWZ19}}]\label{thm:lagPair} Let $T_1$ and $T_2$ be two pointed 4-ended tangles, and let $L = \mc{L}(T_1,T_2)$. Then the Bar-Natan homology is isomorphic to the wrapped Lagrangian intersection Floer homology of the tangle invariants, as bigraded $\F_c[H]$-modules:
	\[\wt{\BN}(L; \F_c)\{-1\} \iso \HF(\wt{\BN}(mT_1; \F_c), \wt{\BN}(T_2; \F_c)).\]
\end{thm}

\section{The proof of \Cref{conj}}\label{sec:proof}

Suppose now that $K$ is a $\vartheta_c$-rational knot. Work of Lewark--Zibrowius identifies $\vartheta_c(K)$ with a certain slope of $\wt{\BN}(T_K; \F_c)$, and this allows us to reduce the proof to a simple statement that can be checked using \Cref{lem:towHomGrad}. Let $\wt{\BN}_a(T; \F_c)$ consist of the non-compact component(s) of $\wt{\BN}(T; \F_c)$.

\begin{prop}[\!\! {\cite[Proposition 6.18]{LZ24}}] \label{prop:ratCurveChar} If $K$ is $\vartheta_c$-rational, then the immersed curve $\wt{\BN}_a(T_K;\F_c)$ is equal to the immersed curve of the rational tangle $Q_n$, for some choice of $n \in 2\Z$, up to some grading shift.
\end{prop}

We have then  $\wt{\BN}_a(T_K; \F_c) = \wt{\BN}(Q_n; \F_c)$, for some $n \in 2\Z$, up to grading shift. The immersed curve invariants $\wt{\BN}(Q_n; \F_c)$ are calculated in \cite{KWZ19}. It turns out that they are independent of the coefficient field, so we may drop it from the notation. These invariants are best described in the following covering space of the 4-punctured sphere:
\[\R^2\setminus ({\scalebox{1}{$\frac{1}{2}$}} \Z)^2 \xra{\alpha} T^2_{4,*} \xra{\beta} S^2_{4,*},\]
where $\beta$ is the double cover given by hyperelliptic involution and $\alpha$ is the universal Abelian cover of the punctured torus. The puncture $\ast$ lifts to the integer lattice $\Z^2 \subset \frac{1}{2}\Z^2$. The lift of $\wt{\BN}(Q_n)$ is (isotopic to) a line of slope $n$, as depicted in \cref{fig:curveLifts} in the cases $n = -2, 0, 2$:

\begin{figure}[h]
\centering
\includegraphics[width=0.58\textwidth]{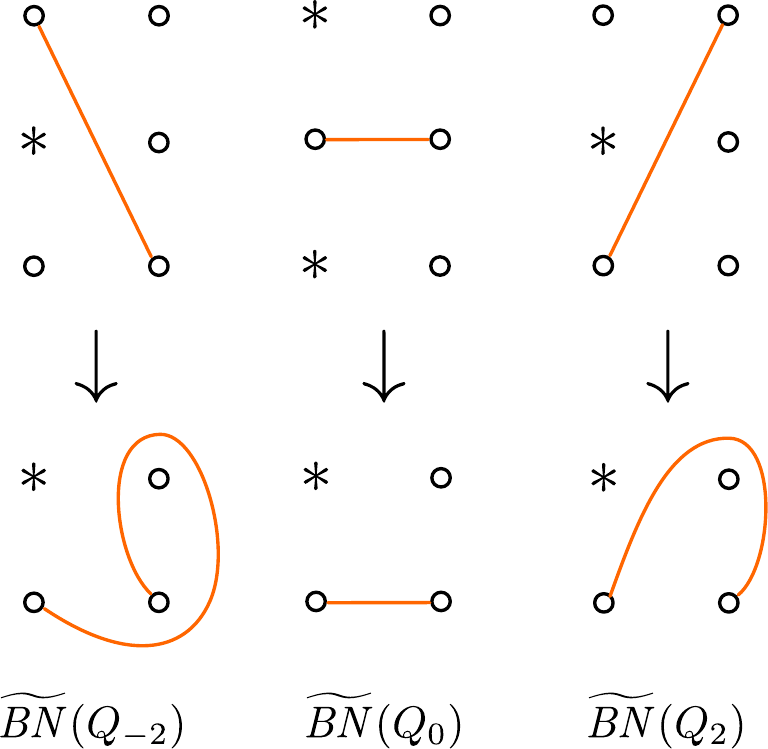}
\caption{Some immersed curve invariants of $Q_n$ and their lifts to the covering space $\R^2\setminus\Z^2$.}
\label{fig:curveLifts}
\end{figure}

\begin{prop}[\!\! {\cite[Corollary 6.14]{LZ24}}] \label{prop:slope} Given a knot $K \subset S^3$, let $\sigma_c$ be the slope of $\wt{\BN}_a(T_K; \F_c)$ near the bottom-right tangle end. Then $\vartheta_c(K) = \ceil{\sigma_c}$.
\end{prop}

Since the curve $\wt{\BN}(Q_n)$ lifts to a curve that is isotopic to a line of slope $n$, the above two propositions reduce the proof of \cref{conj} to proving that $\wt{\BN}_a(T_K; \F_c) = \wt{\BN}(Q_0)$, up to grading shift.  Consider the Bar-Natan homology of the 0-closure $T_K(0)$. Since $T_K$ is obtained by taking the union of a long knot with its Seifert push-off, the closure $T_K(0)$ has linking number 0. Thus, by \Cref{lem:towHomGrad}, the Bar-Natan homology $\wt{\BN}(T_K(0); \F_c)$ has both $\F_c[H]$ towers in grading 0. We may compute this homology using \cref{thm:lagPair}:

\[\begin{split}
	\wt{\BN}(T_K(0))\{-1\}
		&\iso \HF\left(\wt{\BN}(m\oRes), \wt{\BN}(T_K)\right)\\
		&\iso \HF\left(\wt{\BN}(m\oRes), \wt{\BN}_a(T_K)[h]\{q\}\right) \oplus \mathrm{Tors}\\
		&\iso \wt{\BN}(T(2, 2n); \F_c)[h]\{q\} \oplus \mathrm{Tors},
\end{split}\]
where $\mathrm{Tors}$ is a torsion $\F_c[H]$-module, $T(2, 2n)$ is the $(2, 2n)$-torus link and $[h]\{q\}$ is a possible bigrading shift. Clearly both towers of $\wt{\BN}(T_K(0))$ sit in a summand of the homology that is isomorphic to $\wt{\BN}(T(2,2n))$, up to a grading shift. But the homology of 2-strand torus links is well understood --- indeed, we will indicate how to compute it in the next section. In particular, the only way for both towers of $\wt{\BN}(T(2,2n))$ to be in the same homological grading is if $n = 0$. This completes the proof.  \qed \hspace{-3em}

\section{Epilogue}\label{sec:epilogue}

Let us now indicate how to compute $\wt{\BN}(T(2, n); \F_c)$, using a technique that applies more generally and that is the honest source of the proof above. To that end, we will need to look under the hood of \cref{thm:lagPair} and use the bigraded type D structures $\Rd(Q_n; \F_c) \in \cat{Mod}^\mc{B}$. First, we will write $\k$ instead of $\F_c$ in what follows, since the characteristic does not matter and clutters the notation.

\begin{defn} The Bar-Natan algebra $\mc{B}$ is the bigraded path algebra over $\k$ of the quiver
\[\begin{tikzcd}
	\bullet \ar[loop left, "D_\bullet"] \rar[bend right, "S_\bullet"']	&\circ \ar[l, bend right, "S_\circ"'] \ar[loop right, "D_\circ"]
\end{tikzcd},\]
subject to the relations
\[D_\circ S_\bullet = S_\bullet D_\bullet = 0 \hspace{0.5cm}\text{ and } \hspace{0.5cm} D_\bullet S_\circ = S_\circ D_\circ = 0,\]
and with bigrading given by
\[q(1_\odot) = 0, \quad q(S_\odot) = -1, \quad q(D_\odot) = -2, \quad h(1_\odot) = h(S_\odot) = h(D_\odot) = 0,\]
where $\odot \in \{\circ, \bullet\}.$
\end{defn}

\begin{rmk} Alternatively, consider the quiver above as describing an additive category with two objects and with four non-identity morphisms indicated, and suppose that the composites $DS$ and $SD$ vanish. Then the algebra $\mc{B}$ is the collection of all morphisms of this category, where the algebra operation corresponds to composition of morphisms, and we formally set the composite of non-composable morphisms to 0. This is a bigraded category in the sense of Bar-Natan \cite{BN05}.
\end{rmk}

\begin{rmk} By definition, path algebras have idempotent elements $1_\odot$: the constant paths at each vertex. These correspond to identity morphisms in the categorical perspective. The idempotents generate a subring $\mc{I} := \k \gp{1_\circ, 1_\bullet} \iso \k^2$, giving $\mc{B}$ the additional structure of an $\mc{I}$-algebra.
\end{rmk}

Now a type D structure over $\mc{B}$ is, by definition, an $\mc{I}$-module $M$ together with a map $\delta \co M \ra  M \otimes_\mc{I} \mc{B}$ subject to an appropriate ``$d^2=0$" condition:
\[ ( \mathrm{Id}_M \otimes m ) \circ (\delta \otimes \mathrm{Id}_\mc{B}) \circ \delta = 0.\]

\begin{notn} Type D structures are described as labelled directed graphs, with vertices labelled by  $\bullet$ or $\circ$, and edges labelled with elements of $\mc{B}$. The vertices correspond to homogeneous generators (with respect to the action of $\mc{I}$) and the edges are the homogeneous components of the differential $\delta$. To avoid heavy use of brackets, we denote homological and quantum shifts by subscripts and left-superscripts, respectively. For example
\[\prescript{q}{}{\bullet_h}\]
is a type D structure generator fixed by $1_\bullet$ and in (homological, quantum)-bigrading $(h, q)$.
\end{notn}

The $\Rd$-invariants of $Q_n$ are explicitly computed as Example 4.27 of \cite{KWZ19} (where $Q_n$ is oriented compatibly with the 0-closure): $\Rd(Q_0) = \prescript{0}{}{\bullet_0}$ and, more generally,
\[\Rd(Q_n; k) =
\begin{cases}
	\underbrace{\begin{tikzcd}[cramped, sep = small, ampersand replacement = \&]
			\prescript{3n-1}{}{\circ_{n}} \rar["X"]	\&\cdots \rar["D"]	\&\circ\rar["SS"]	\&\circ\rar["D"]	\& \circ \rar["S"] \& \prescript{n}{}{\bullet_0}
		\end{tikzcd}}_{-n+1}		& \text{ if } n < 0\\
	\underbrace{\begin{tikzcd}[cramped, sep = small, ampersand replacement = \&]
			\prescript{n}{}{\bullet_0} \rar["S"]	\&\circ \rar["D"]	\&\circ\rar["SS"]	\&\circ\rar["D"]	\& \cdots \rar["X"]	\&\prescript{3n-1}{}{\circ_n}
		\end{tikzcd}}_{n+1}		& \text{ if } n > 0,\\
\end{cases}\]
where the algebra element $X$ is $D$ if $n$ is even and $SS$ if $n$ is odd.

Finally, the following element is defined in $\mc{B}$:
\[H := \mathit{SS}_\bullet - D_\bullet + \mathit{SS}_\circ - D_\circ.\]
This gives the Bar-Natan algebra the structure of a $\k[H]$-algebra, and, by design, this structure is compatible with the $\k[H]$-module structure of Bar-Natan homology:

\begin{thm}[\!\! {\cite[Proposition 4.31]{KWZ19}}] \label{thm:hom} Let $T_1$ and $T_2$ be two pointed oriented 4-ended tangles. Then there is a homotopy
\begin{equation} \label{eq:hom} 
\wt{\CBN}(\mc{L}; \k)\{-1\} \htp \Mor(\Rd(mT_1; \k), \Rd(T_2; \k))
\end{equation}
of bigraded chain complexes of $\k[H]$-modules, where $m$ denotes the mirror, and the bifunctor $\Mor(-, -)$ above is the internal Hom in the category of bigraded type D structures.
\end{thm}

The type D structure of $\Mor(\Rd_1, \Rd_2)$ is defined in \cite[\S2]{KWZ19}. Briefly, $\Mor(\Rd_1, \Rd_2)$ consists of all morphisms $\Rd_1 \ra \Rd_2$, not just the grading preserving ones. Given generators $x_i \in \Rd_i$ the quantum and homological grading of a morphism is given by
\[\gr(x_1 \xra{f} x_2) = \gr(x_2) - gr(x_1) + \gr(f).\]
Finally, a differential $D$ on $\Mor(\Rd_1, \Rd_2)$ is given on morphisms between generators by pre- and post-composing with the $\delta_i$ differentials on $\Rd_i$:
\[D(x_1 \xra{f} x_2) = f \circ \delta_1 -  \delta_2 \circ f.\]
For our purposes, note the following computations:
\[
\Mor(\prescript{i}{}{\bullet_j}, \prescript{k}{}{\circ_l}) 	=  \k[H]\gp{\prescript{i}{}{\bullet_j} \xra{S_\bullet} \prescript{k}{}{\circ_l}} \iso 	\prescript{k-i-1}{}{\big(\k[H]\big)_{l-j}}
\]
\[ 
\Mor(\prescript{i}{}{\bullet_j}, \prescript{k}{}{\bullet_l}) 	=	\k[H]\gp{\prescript{i}{}{\bullet_j} \xra{1_\bullet} \prescript{k}{}{\bullet_l}, \prescript{i}{}{\bullet_j} \xra{D_\bullet} \prescript{k}{}{\bullet_l}} 	\iso 	\prescript{k-i}{}{\big(\k[H]\big)_{l-j}} \oplus \prescript{k-i-2}{}{\big(\k[H]\big)_{l-j}}
\]

To give the simplest application of \cref{thm:hom}, the unknot $U$ is $\mc{L}(\oRes, \iRes)$. We have thus
\[\wt{\BN}(U)\{-1\} \iso H_* \left[ \Mor(\prescript{0}{}{\bullet_0}, \prescript{0}{}{\circ_0}) \right] \iso \prescript{-1}{}{\big(\k[H]\big)_0}.\]
		
Now we can give a rapid computation of $\wt{\BN}(T(2,n)) = \wt{\BN}(\mc{L}(\oRes, Q_n))$. If $n < 0$, then
\[\begin{split}
\wt{\BN}(T(2,n))\{-1\} 
&\iso H_*\left[ 
	\Mor(\prescript{0}{}{\bullet_0}, 
		\begin{tikzcd}[cramped, sep = small, ampersand replacement = \&]
			\prescript{3n-1}{}{\circ_n} \rar["X"]	\&\circ \rar["D"]	\&\circ\rar["SS"]	\&\cdots \rar	\&\prescript{n}{}{\bullet_0}
		\end{tikzcd})
		\right]\\
&\iso H_* \left[
	\begin{tikzcd}[cramped, sep = small, ampersand replacement = \&]
		\Mor(\prescript{0}{}{\bullet_0},\prescript{3n-1}{}{\circ_n}) \rar["X_*"]	\&\Mor(\prescript{0}{}{\bullet_0}, \circ) \rar["D_*"]	\&\Mor(\prescript{0}{}{\bullet_0},\circ)\rar["SS_*"]	\& \cdots \rar	\&\Mor(\prescript{0}{}{\bullet_0}, \prescript{n}{}{\bullet_0})
	\end{tikzcd}
	\right],
\end{split}\]
where the maps above are the ones induced by postcomposing with the components of the differential on $\Rd(Q_n)$. It is convenient to organize the above complex in a grid as follows:
\begin{center}
\begin{tikzpicture}
	\draw[step=1cm,gray,very thin] (-5.5,-5.5) grid (-2.5,-2.5);
	\draw[step=1cm,gray,very thin] (-5.5,-1.5) grid (-2.5, 4.5);
	\draw[step=1cm,gray,very thin] (-1.5,-5.5) grid (4.5, -2.5);
	\draw[step=1cm,gray,very thin] (-1.5,-1.5) grid (4.5,4.5);
	\draw[thick, ->] (-5.5, -5) -- (4.5, -5);
	\draw[thick, ->] (-5, -5.5) -- (-5, 4.5);
	 \foreach \x in {-4.5,-3.5}
   		\draw (\x, \x) node {$\k[H]$};
	\draw[thick, ->] (-4.25, -4.25) -- (-3.75, -3.75);
	\draw (-4, -4) node[anchor=south east] {$X_*$};
	\draw (-2, -2) node {$\iddots$};
	\foreach \x in {-0.5, 0.5, 1.5, 2.5, 3.5}
		\draw (\x, \x) node {$\k[H]$};
	\draw (3.68, 2.5) node {$\k[H]$};
	\foreach \x in {-1.25, 0.75, 2.75}
	{
		\draw[thick, ->] (\x,\x) -- +(0.5, 0.5);
		\draw (\x,\x) ++(0.25, 0.25) node[anchor=south east] {$\times H$};
	}
	\draw[thick, ->] (2.9, 2.5) -- (3.26, 2.5);
	\draw (3.07, 2.5) node[anchor=south] {$1$};
	\foreach \x in {-4, -3, -2, -1, 0}
		\draw (\x, -5.4) ++(3.5, 0) node {\x};
	\draw (-4.5, -5.4) node {$n$};
	\draw (-3.5, -5.4) node {$n+1$};
	\draw (-5.6, 3.5) node {$n$};
	\draw (-5.6, 2.5) node {$n-2$};
	\draw (-5.6, 1.5) node {$n-4$};
	\draw (-5.6, 0.5) node {$n-6$};
	\draw (-5.6, -0.5) node {$n-8$};
	\draw (-5.6, -3.5) node {$3n$};
	\draw (-5.65, -4.5) node {$3n-2$};
\end{tikzpicture}
\end{center}

Where the horizontal and vertical axes measure the homological and quantum grading, respectively, and where only the nonzero components of the differential are indicated. These components are easy to compute: every morphism group, except for the last one, is generated over $\k[H]$ by an $S_\bullet$, which $D_*$ takes to 0 and $SS_*$ takes to $SSS = HS$. The last morphism group is generated by $1_\bullet$ and $D_\bullet$ and the incoming differential is $S_\bullet \mapsto S_\circ S_\bullet = H1_\bullet + D_\bullet$. 

Taking homology of the above bigraded complex of free $\k[H]$-modules yields $\wt{\BN}(T(2,n);\k)$. In particular, when $n$ is even, the two towers are in homological grading $n$ and $0$, in accordance with \Cref{lem:towHomGrad}. The computation for $n\geq0$ is analogous.

\bibliographystyle{alphaurl}
\bibliography{Remark_on_LewZib.bbl} 

\end{document}